\documentclass[11pt, onesided, leqno]{article}
\usepackage[top=20truemm,bottom=20truemm,left=20truemm,right=20truemm]{geometry}
\usepackage{enumerate,enumitem}
\usepackage[all]{xy}
\usepackage{mathtools}
\usepackage{amssymb,amsthm,enumerate,color}
\usepackage{amsmath,mathpazo,fancybox,ascmac,amsfonts,bm}
\usepackage{pb-diagram}
\RequirePackage[dvipdfmx]{graphicx}
\usepackage{tikz}
\usetikzlibrary{intersections, calc}
\usetikzlibrary{decorations.markings}
\usetikzlibrary{arrows.meta}
\tikzset{
  arrows along my path/.style={
    postaction={
      decorate,
      decoration={
        markings,
        mark=between positions 0.03 and 1 step 10pt with {\arrow{Stealth[length=5pt]}},
   }}}}
   
\newcommand{\doublespace}
   {\addtolength{\baselineskip}{0.25\baselineskip}}

\numberwithin{equation}{section}
\theoremstyle{plain}
\newtheorem{thm}{Theorem}[section]

\newtheorem{cor}[thm]{Corollary}
\newtheorem{ex}[thm]{Example}
\newtheorem{lem}[thm]{Lemma}

\newtheorem{prop}[thm]{Proposition}

\theoremstyle{definition}
\newtheorem{defn}[thm]{Definition}
\theoremstyle{remark}
\newtheorem{rem}[thm]{Remark}

\newcommand{\C}{\mathbb{C}}
\newcommand{\R}{\mathbb{R}}

\newcommand{\N}{\mathbb{N}}

\newcommand{\calL}{\mathcal{L}}

\newcommand{\calQ}{\mathcal{Q}}
\newcommand{\calP}{\mathcal{P}}

\let\Re\relax
\DeclareMathOperator{\Re}{{\sf Re}}
\let\Im\relax
\DeclareMathOperator{\Im}{{\sf Im}}

\newcommand{\dd}{{\rm d}}   
\newcommand\MP{\mathbf{MP}}  
\newcommand\ND{N(0,1)} 
\newcommand{\bb}{\mathbf{b}} 
\newcommand\bern{B(0,1)}
\newcommand{\FS}{{\mathbf{f}_{1/2}}} 
\renewcommand{\epsilon}{\varepsilon}  

\title{On Boolean selfdecomposable distributions}
\author{Takahiro Hasebe\thanks{Department of Mathematics, Hokkaido University, Kita 10, Nishi 8, Kita-Ku, Sapporo, Hokkaido, 060-0810, Japan (E-mail: thasebe@math.sci.hokudai.ac.jp)} \and 
Kei Noba\thanks{School of Statistical Thinking, The Institute of Statistical Mathematics, Midori-cho 10-3, Tachikawa, Tokyo, 190-8562, Japan (E-mail: knoba@ism.ac.jp)} \and
Noriyoshi Sakuma\thanks{Graduate School of Natural Sciences, Nagoya City University, Mizuho-ku, Nagoya, Aichi, 467-8501, Japan (E-mail: sakuma@nsc.nagoya.ac.jp)} \and
Yuki Ueda\thanks{Department of Mathematics, Hokkaido University of Education, Hokumon-cho 9, Asahikawa, Hokkaido, 070-8621, Japan (E-mail: ueda.yuki@a.hokkyodai.ac.jp)}}

\begin{document}

\maketitle  
\doublespace

\begin{abstract}
This paper introduces the class of selfdecomposable distributions concerning Boolean convolution. A general regularity property of Boolean selfdecomposable distributions is established; in particular the number of atoms is at most two and the singular continuous part is zero. We then analyze how shifting probability measures changes Boolean selfdecomposability. Several examples are presented to supplement the above results. Finally, we prove that the standard normal distribution $N(0,1)$ is Boolean selfdecomposable but the shifted one $N(m,1)$ is not for sufficiently large $|m|$.  
\end{abstract}

\section{Introduction}

In non-commutative probability theory, random variables are defined as elements in some $\ast$-algebra. A remarkable aspect of this theory is that various notions of independence exist for those random variables. From a certain viewpoint, notions of independence are classified into five ones: tensor, free, Boolean, monotone and anti-monotone independences (see \cite{Mur}). Further, each notion of independence associates the convolution of (Borel) probability measures on $\R$ that is defined to be the distribution of the sum of two independent (self-adjoint) random variables having prescribed distributions. 

Limit theorems have been central subjects in both commutative (classical) and non-commutative probability theories. Among all, Khintchine introduced the class $\calL$ of limit laws of certain independent triangular arrays. More precisely, a probability measure $\mu$ on $\R$ belongs to $\calL$ if there exist a sequence of independent $\R$-valued random variables $\{X_n\}_{n\ge1}$ and sequences of deterministic numbers $\{a_n\}_{n\ge0} \subset \R$ and $\{b_n\}_{n\ge1} \subset (0,\infty)$ such that 
\begin{itemize}
\item the family $\{b_ n X_k\}_{n \ge k \ge 1}$ forms an infinitesimal triangular array, i.e.\ 
\[
\lim_{n\to\infty}\sup_{1 \le k \le n} P[|b_n X_k| \ge \epsilon]=0  \qquad \text{for all} \qquad \epsilon>0, 
\] 
\item the law of the sequence of random variables 
\[
S_n = a_n+ b_n(X_1+ X_2 + \cdots + X_n), \qquad n \ge 1, 
\]
converges weakly to $\mu$. 
\end{itemize} 
When $\{X_n\}_{n\ge1}$ is furthermore identically distributed, the limit distribution $\mu$, if exists, is known to be a stable distribution. Thus the class $\calL$ contains the class of stable distributions as a subset.  
In 1937, L\'{e}vy characterized the class $\calL$ by the following property (see e.g.\ \cite[Theorem 1, Section 29]{GK54}): A probability measure $\mu$ on $\R$ is in $\calL$, if and only if $\mu$ is {\it selfdecomposable}, i.e., for any $c\in (0,1)$, there exists a probability measure $\mu_c$ on $\R$, such that $\mu=D_c(\mu)\ast \mu_c$, where $D_c(\mu)$ is the push-forward of $\mu$ by the mapping $x\mapsto cx$, and $\ast$ denotes the classical convolution (see e.g. \cite{Sato}). 

In non-commutative probability theory, an analogous limit theorem can be formulated for each notion of independence. Bercovici and Pata \cite{BP00}, Chistyakov and Goetze \cite{CG08} and Wang \cite{Wang08} proved a parallelism between classical Khintchine's limit theorem above and its non-commutative versions  (for further details see Subsection \ref{subsec:BP} below).  Correspondingly, notions of selfdecomposable distributions with respect to other convolutions are also defined (see Definition \ref{defn:SD} below). In particular, the two classes of freely and monotonically selfdecomposable distributions were investigated in details (see e.g. \cite{BNT02, HST19, HT2016} and \cite{FHS20}, respectively), where analogy and disanalogy with the classical class $\calL$ are discussed. By contrast, little had been done on Boolean selfdecomposable distributions. Although our definition of Boolean selfdecomposable distributions is so natural for specialists, there had been no formal definition in the literature to the authors' best knowledge.  

The purpose of this paper is to study the class of Boolean selfdecomposable distributions. Results will be presented along the following lines.  In Section 2, we introduce some concepts and some preliminary results that are used in the remainder of the paper.  In Section 3, we establish a general regularity result on Boolean selfdecomposable distributions. Especially, we demonstrate that the Boolean selfdecompsable distributions have at most two atoms and do not have singular continuous part. We then investigate how Boolean selfdecomposability changes under shifts. It turns out that Boolean selfdecomposability is typically broken when the distribution is shifted with a sufficiently large positive or negative number. 
Furthermore, we observe some distributional properties of Boolean selfdecomposable distributions through several examples. In Section 4, we study Boolean selfdecomposability for normal distributions. The results in this section are motivated by the fact that every normal distribution is both classically and freely selfdecomposable (see e.g.\ \cite{Sato} and \cite{HST19}, respectively). Our conclusion is that the standard normal distribution $N(0,1)$ is Boolean selfdecomposable too; however, the shifted one $N(m,1)$ is not Boolean selfdecomposable when $|m|$ is sufficiently large as a consequence of the aforementioned general result in Section 3. Simulations suggest that $|m| \simeq 3.09$ is the approximate threshold.

\section{Preliminaries}
In the first of this section, we shall introduce analytic tools for understanding Boolean and free additive convolutions. After that we summarize a transfer principle for limit theorems for convolutions, which induces a bijection between different kinds of infinitely divisible distributions.

\subsection{Boolean convolution and analytic tools}

The Boolean convolution of probability measures on $\R$ was defined in \cite{SW97}. We set $\calP(\R)$ to be the set of all Borel probability measures on $\R$, $\C^+$ the complex upper half-plane and $\C^-$ the complex lower half-plane. For $\mu\in \calP(\R)$, we define the \emph{Cauchy transform} $G_\mu\colon\C^{+}\to \C^{-}$ and the \emph{$F$-transform} $F_\mu\colon\C^{+}\to \C^{+}$ as follows: 
\begin{align*}
G_\mu(z):=\int_\R \frac{1}{z-x}\mu({\rm d}x) \qquad \text{and} \qquad F_\mu(z):=\frac{1}{G_\mu(z)}, \qquad z\in \C^+.
\end{align*}

The {\it self-energy function $K_\mu$ of $\mu$} is defined by
\begin{align*}
K_\mu(z):=z-F_\mu(z),\qquad z\in \C^+.
\end{align*}
Since $\Im[F_\mu(z)]\ge \Im[z]$ for $z\in \C^+$  (see \cite[Corollary 5.3]{BV93}), the function $K_\mu$ is an analytic map from $\C^+$ to $\C^- \cup \R$. The self-energy functions are characterized in the following way (see \cite[Proposition 3.1]{SW97} for further details).
\begin{prop}\label{lem:K} Let $K$ be an analytic function on $\C^+$. The following assertions are equivalent:
\begin{enumerate}
\item[\rm (1)] There exists $\mu\in \calP(\R)$ such that $K=K_\mu$.
\item[\rm (2)] There exist $b\in \R$ and a finite positive measure $\tau$ on $\R$, such that $K$ has the form
\begin{align}\label{K-function}
K(z)=b+\int_\R \frac{1+xz}{z-x}\tau({\rm d}x), \qquad z\in \C^+.
\end{align}
\end{enumerate}
If either of $(1)$ or $\rm (2)$ holds, then the pair $(b,\tau)$ is uniquely determined by $\mu$ (from formulas \eqref{eq:K1}--\eqref{eq:K2} below).   
\end{prop}

We call such a pair $(b,\tau)$ the {\it Boolean generating pair} for $\mu$. By Lemma \ref{K-function}, there is a one-to-one correspondence between the set $\calP(\R)$ and the set of all Boolean generating pairs. Hence, we denote by $\mu_{\uplus}^{(b,\tau)}$ the probability measure $\mu$ with a Boolean generating pair $(b,\tau)$.   The Boolean generating pair for $\mu_{\uplus}^{(b,\tau)}$ can be computed from the following formulas: 
\begin{align}
K_{\mu_{\uplus}^{(b,\tau)}}(i) &= b - i \tau(\R), \label{eq:K1}\\
\int_\alpha^\beta (1+x^2) \, \tau({\rm d}x) &= - \frac{1}{\pi} \lim_{\epsilon\to0^+}\int_\alpha^\beta \Im[K_{\mu_{\uplus}^{(b,\tau)}}(x+i\epsilon)] \,{\rm d}x \label{eq:K2}
\end{align}
for all $\alpha,\beta \in \R$ that are continuity points of $\tau$. 
The latter formula is referred to as the \emph{Stieltjes inversion formula}. 

For $\mu,\nu\in \calP(\R)$, their {\it Boolean convolution} $\mu\uplus \nu$ is the probability measure on $\R$ satisfying 
\begin{align}\label{boolean convolution}
K_{\mu\uplus \nu}(z)=K_{\mu}(z)+K_{\nu}(z), \qquad z\in \C^+.
\end{align}

The {\it $\eta$-transform} of $\mu\in\calP(\R)$ is defined by
\begin{align*}
\eta_\mu(z):=1-zF_\mu\left( \frac{1}{z}\right), \qquad z\in \C^-.
\end{align*}
It is obvious to see that $\eta_\mu(z)=zK_\mu (1/z)$, and therefore $\eta_{\mu\uplus \nu}(z)=\eta_\mu(z)+\eta_\nu(z)$ for all $z\in\C^-$. 

 A probability measure $\mu$ is said to be {\it Boolean infinitely divisible} if for each $n\in\N$ there exists $\mu_n\in \calP(\R)$ such that $\mu=\underbrace{\mu_n \uplus \cdots \uplus \mu_n}_{n\text{ times}}$. For each $\mu\in \calP(\R)$ with a Boolean generating pair $(b,\tau)$, let us set $\mu_n$ as the probability measure with Boolean generating pair $(b/n, \tau/n)$ for each $n\in \N$. Then $\mu$ is the $n$-fold Boolean convolution of $\mu_n$. Therefore every probability measure on $\R$ is Boolean infinitely divisible.

For a probability measure $\mu$ on $\R$ with a Boolean generating pair $(b,\tau)$, let us set
\begin{equation}\label{relation}
\begin{cases}
a=\tau(\{0\}),\\[2mm]
\nu({\rm d}x)=\dfrac{1+x^2}{x^2} \cdot \mathbf{1}_{\R\setminus\{0\}}(x) \tau ({\rm d}x),\\[2mm]
\displaystyle \gamma=b +\int_\R x\left(\mathbf{1}_{[-1,1]}(x)-\dfrac{1}{1+x^2}\right) \nu({\rm d}x).
\end{cases}
\end{equation}
The triplet $(a,\nu,\gamma)$ thus defined fulfills 
\begin{enumerate}[label=\rm(T)]
\item\label{item:T} $a \ge 0$, $\gamma \in \R$ and $\nu$ is a positive measure on $\R$ such that $\nu(\{0\})=0$ and $\int_\R (1\land x^2) \nu({\rm d}x)<\infty$. 
\end{enumerate}
Moreover, the set of triplets $(a,\nu,\gamma)$ satisfying \ref{item:T} is in bijection with the set of pairs $(b,\tau)$ of a real number $b$ and a finite positive measure $\tau$ on $\R$.  
In terms of this bijection, formula \eqref{K-function} has the following equivalent form
\begin{align}\label{BLKeta}
\eta_\mu(z)=zK_\mu\left(\frac{1}{z}\right)=\gamma z +a z^2+ \int_\R \left(\frac{1}{1-zx}-1-zx\mathbf{1}_{[-1,1]}(x)\right) \nu({\rm d}x).
\end{align}
The real number $a\ge 0$ is called the {\it Boolean Gaussian component} for $\mu$, and the measure $\nu$ is called the  {\it Boolean L\'{e}vy measure} for $\mu$.  The triplet $(a,\nu,\gamma)$ is called the \emph{Boolean L\'evy triplet}. 
 
 Let $\mu \in \calP(\R)$ and $\nu$ be its Boolean L\'evy measure. Let $\nu^{\rm ac}$ be the Lebesgue absolutely continuous part of $\nu$. We introduce functions $k_\mu, \ell_\mu\colon\R\setminus\{0\}\rightarrow[0,\infty)$ by 
\begin{align}\label{boolean Levy measure}
k_\mu(x) = |x| \frac{{\rm d} \nu^{\rm ac}}{{\rm d} x} (x) \qquad \text{and} \qquad \ell_\mu(x) = |x| k_\mu(x), \qquad  x \in \R \setminus\{0\}. 
\end{align}
Those functions are key ingredients in our main results.  Then $\nu^{\rm ac} ({\rm d}x)$ can be expressed in the form $x^{-2}\ell_\mu(x)\,{\rm d}x = |x|^{-1}k_\mu(x)\,{\rm d}x$.

\subsection{Free convolution and analytic tools}

The free convolution of general probability measures on $\R$ was defined in \cite{BV93}. According to \cite[Proposition 5.4]{BV93}, for any $\mu\in \calP(\R)$ and any $\lambda>0$ there exist $\alpha,\beta,M>0$ such that $F_\mu$ is univalent on the set $\Gamma_{\alpha,\beta}:=\{z\in \C^+: \Im [z]>\beta, |\Re [z]|<\alpha \Im [z]\}$ and $\Gamma_{\lambda,M}\subset F_\mu(\Gamma_{\alpha,\beta})$. This implies that the right inverse function $F_\mu^{-1}$ of $F_\mu$ exists on $\Gamma_{\lambda,M}$. The {\it Voiculescu transform} $\varphi_\mu$ is defined by
\begin{align*}
\varphi_\mu(z):=F_\mu^{-1}(z)-z, \qquad z\in \Gamma_{\lambda, M}.
\end{align*}
For $\mu,\nu\in \calP(\R)$, their {\it free convolution} $\mu\boxplus\nu$ is the probability measure satisfying
\begin{align}
\varphi_{\mu\boxplus \nu}(z)=\varphi_\mu(z)+\varphi_\nu(z)
\end{align}
for $z$ in the common domain in which the three transforms are defined.

The {\it R-transform} of $\mu\in \calP(\R)$ is defined by
\begin{align*}
R_\mu(z):=z\varphi_\mu\left(\frac{1}{z} \right), \qquad z\in \C^- \text{ with } \frac{1}{z}\in \Gamma_{\lambda,M}.
\end{align*}
By definition, it is obvious to see that $R_{\mu\boxplus \nu}(z)=R_\mu(z)+R_\nu(z)$ for all $z$ in the common domain in which the three transforms are defined.

A probability measure $\mu$ on $\R$ is said to be {\it freely infinitely divisible} (denoted by $\mu \in I(\boxplus)$) if for each $n\in \N$ there exists $\mu_n\in \calP(\R)$ such that $\mu=\underbrace{\mu_n \boxplus \cdots \boxplus \mu_n}_{n \text{ times}}$. 
Several criteria for free infinite divisibility were given by using harmonic analysis, complex analysis, and combinatorics (see e.g. \cite{BBLS11, AH13}). In particular, the following characterization of $I(\boxplus)$ is well-known (see \cite[Theorem 5.10]{BV93}). 
\begin{prop}
For $\mu \in \calP(\R)$, the following conditions are equivalent.
\begin{enumerate}
\item[\rm (1)] $\mu \in I(\boxplus)$.
\item[\rm (2)] The Voiculescu transform $\varphi_\mu$ has an analytic extension (denoted by the same symbol $\varphi_\mu$) defined on $\C^+$ with values in $\C^-\cup \R$. 
\item[\rm (3)]  There exist $b\in\R$ and a finite positive measure $\tau$ on $\R$ such that 
\begin{align}\label{Voiculescu}
\varphi_\mu(z)=b+\int_\R \frac{1+xz}{z-x}\tau({\rm d}x), \qquad z\in\C^+. 
\end{align}

\end{enumerate}
Note that such a pair $(b,\tau)$ is uniquely determined by $\mu$ for the same reason as \eqref{eq:K1}--\eqref{eq:K2}. Conversely, given $b \in \R$ and  a finite positive measure $\tau$ on $\R$, there exists $\mu \in I(\boxplus)$ such that \eqref{Voiculescu} holds. 
\end{prop}

The above pair $(b,\tau)$ is called the {\it free generating pair} for $\mu$, and we denote by $\mu_{\boxplus}^{(b,\tau)}$ the freely infinitely divisible distribution with a free generating pair $(b,\tau)$.

For $\mu\in I(\boxplus)$ with a free generating pair $(b,\tau)$, formula \eqref{Voiculescu} is equivalent to
\begin{align}\label{FLKR}
R_\mu(z)=\gamma z + a z^2 + \int_\R \left(\frac{1}{1-zx}-1-zx\mathbf{1}_{[-1,1]}(x)\right) \nu({\rm d}x), \qquad z\in \C^-,
\end{align}
where the triplet $(a,\nu,\gamma)$ is given by \eqref{relation}, see \cite{BNT02} for further details. The real number $a\ge0$ is called the {\it free Gaussian component} for $\mu$, and the measure $\nu$ is called the {\it free L\'{e}vy measure} for $\mu$. 





\subsection{Boolean-to-free Bercovici-Pata bijection}\label{subsec:BP}

In \cite[Theorem 6.3]{BP}, Bercovici and Pata found a remarkable equivalence between limit theorems for convolutions $\ast$, $\boxplus$ and $\uplus$. It leads to a bijection between the corresponding three classes of infinitely divisible distributions. Bercovici and Pata's work was concerning limit theorems for i.i.d.\ random variables. We present here a statement in a generalized setting of infinitesimal triangular arrays, combining \cite[Theorem 1]{BP00}, \cite[Theorem 2.1]{CG08} and \cite[Theorem 5.3]{Wang08}. 
We use the notation $\mu_n \xrightarrow{w} \mu$ to mean that $\mu_n$ converges weakly to $\mu$ when $\mu_n,\mu$ are finite positive measures on $\R$.

\begin{thm}\label{BP} 
Let $\{a_n\}_{n\ge1}$ be a sequence of real numbers and $\{\mu_{n,k}\}_{1 \le k \le k_n, 1 \le n}$ be a family (or an array) of probability measures on $\R$ such that $\{k_n\}_{n\ge1}$ is a sequence of natural numbers which tends to infinity and 
\begin{equation}\label{eq:infinitesimal}
\lim_{n\to\infty} \sup_{1 \le k \le k_n}\mu_{n,k}(\R \setminus (-\epsilon, \epsilon)) =0 \qquad  \text{for all} \qquad \epsilon>0. 
\end{equation}
We set 
\[
a_{n,k} = \int_{|x|<1} x \, \mu_{n,k}({\rm d}x) \qquad \text{and} \qquad \mathring\mu_{n,k}(B) = \mu_{n,k}(B+a_{n,k}) 
\]
for Borel subsets $B$ of $\R$. 
Then the following conditions are equivalent. 
\begin{enumerate}
\item $\delta_{a_n} \uplus \mu_{n,1} \uplus  \mu_{n,2} \uplus \cdots \uplus \mu_{n,k_n}\xrightarrow{w} \mu$ for some $\mu \in \calP(\R)$ as $n\rightarrow\infty$; 
\item $\delta_{a_n} \boxplus \mu_{n,1} \boxplus  \mu_{n,2} \boxplus \cdots \boxplus \mu_{n,k_n}\xrightarrow{w} \mu'$ for some $\mu' \in \calP(\R)$ as $n\rightarrow\infty$; 
\item $\delta_{a_n} \ast \mu_{n,1} \ast  \mu_{n,2} \ast \cdots \ast \mu_{n,k_n}\xrightarrow{w} \mu''$ for some $\mu'' \in \calP(\R)$ as $n\rightarrow\infty$; 
\item  there exist  $b \in \R$ and a positive finite measure $\tau$ on $\R$  such that 
\begin{align*}
&\lim_{n\rightarrow \infty} \left[ a_n+ \sum_{k=1}^{k_n} \left( a_{n,k}+ \int_{\R}\frac{x}{x^2+1}\mathring\mu_{n,k}({\rm d}x)\right) \right]=b \qquad\text{and} \\
&\sum_{k=1}^{k_n}\frac{x^2}{x^2+1}\mathring\mu_{n,k}({\rm d}x) \xrightarrow{w}  \tau({\rm d}x) \qquad \text{as} \qquad n\to\infty. 
\end{align*}
\end{enumerate}
Moreover, if one of those statements holds then $\mu = \mu_{\uplus}^{(b,\tau)}, \mu' = \mu_{\boxplus}^{(b,\tau)}$ and $\mu''=\mu_{\ast}^{(b,\tau)}$, where the last measure is the infinitely divisible distribution characterized by the L\'evy--Khintchine representation
\[
\int_\R e^{z x} \mu_{\ast}^{(b,\tau)}({\rm d}x) = \exp\left[b z  +  \int_\R \left(e^{zx} - 1 - \frac{zx}{1+x^2} \right) \frac{1+x^2}{x^2}\, \tau({\rm d}x)  \right], \qquad z \in i \R. 
\] 
\end{thm}
From this result, it is natural to identify the limit distributions $\mu,\mu',\mu''$. For later use, we only formulate the map $\Lambda_B$ between the two classes $\calP(\R)$ and $I(\boxplus)$ sending $\mu_{\uplus}^{(b,\tau)}\in \calP(\R)$ to $\mu_{\boxplus}^{(b,\tau)}\in I(\boxplus)$ for each $b\in \R$ and each finite positive measure $\tau$ on $\R$. This map $\Lambda_B$ is obviously a bijection and is called the {\it Boolean-to-free Bercovici-Pata bijection}.  It turns out that $\Lambda_B$ has the following properties. 

\begin{lem}\label{lem:BFBP} Let $\mu_1,\mu_2,\mu\in \calP(\R)$ and $c\in \R$.
\begin{enumerate}
\item\label{BFBP:R_eta} $R_{\Lambda_B(\mu)}=\eta_\mu$ on $\C^-$. 
\item\label{BFBP:hom} $\Lambda_B(\mu_1\uplus\mu_2)=\Lambda_B(\mu_1) \boxplus \Lambda_B(\mu_2)$.
\item\label{BFBP:dilation} $\Lambda_B(D_c(\mu))=D_c(\Lambda_B(\mu))$. 
\item\label{BFBP:dirac} $\Lambda_B(\delta_c)=\delta_c$.
\item\label{BFBP:weak} $\Lambda_B$ is a homeomorphism with respect to weak convergence.   
\end{enumerate}
\end{lem}
Formulas \eqref{BFBP:R_eta}--\eqref{BFBP:dirac} can be checked by straightforward algebraic calculations, while assertion \eqref{BFBP:weak} needs careful analysis. 
The reader is referred to \cite[(2.20)]{BN08} for formula \eqref{BFBP:R_eta}  and to \cite{BNT02} for the other assertions.

Theorem \ref{BP} implies the following equivalence of L\'evy's limit theorems (the classical one is already mentioned in Introduction). Although this will not be directly used in our paper, this will serve as a good motivation for studying Boolean selfdecomposable distributions.

\begin{cor}\label{cor:L} Let $\{a_n\}_{n\ge1}$ be a sequence of real numbers, $\{b_n\}_{n\ge1}$ a sequence of positive real numbers and $\{\mu_{n}\}_{n \ge1}$ a sequence of probability measures on $\R$ such that \eqref{eq:infinitesimal} is fulfilled for the array $\mu_{n,k}:= D_{b_n}(\mu_k), 1 \le k \le n, 1 \le n$. 
Then the following conditions are equivalent. 
\begin{enumerate}  
\item $\delta_{a_n} \uplus D_{b_n}(\mu_{1} \uplus  \mu_{2} \uplus \cdots \uplus \mu_{n}) \xrightarrow{w} \mu$ for some $\mu \in \calP(\R)$ as $n\rightarrow\infty$; 
\item $\delta_{a_n} \boxplus D_{b_n}(\mu_{1} \boxplus  \mu_{2} \boxplus \cdots \boxplus \mu_{n}) \xrightarrow{w} \mu'$ for some $\mu' \in \calP(\R)$ as $n\rightarrow\infty$; 
\item $\delta_{a_n} \ast D_{b_n}(\mu_{1} \ast  \mu_{2} \ast \cdots \ast \mu_{n}) \xrightarrow{w} \mu''$ for some $\mu'' \in \calP(\R)$ as $n\rightarrow\infty$. 
\end{enumerate}
\end{cor}

The possible limit distributions $\mu,\mu', \mu''$ in Corollary \ref{cor:L} belong to certain subclasses of freely, Boolean and classically infinitely divisible distributions, respectively. Those subclasses are all characterized by the respective notions of \emph{selfdecomposability} defined later in Definition \ref{defn:SD}. This fact follows e.g.\ from the last part of Theorem \ref{BP}, \cite[Theorem 1, Section 29]{GK54} and \cite[Theorem 4.8]{BNT02} (or \cite[Theorem 2.10]{CG08}) for $\ast$ and $\boxplus$. The case $\uplus$ can be treated similarly to (actually easier than) the free case. See also the paragraph following Definition \ref{defn:SD}.


\section{Boolean selfdecomposable distributions}\label{sec:SD}

\subsection{The class of Boolean selfdecomposable distributions}

The classical notion of selfdecomposability can be extended to general convolutions of probability measures in the following way. 

\begin{defn}\label{defn:SD} Let $\circ$ be a binary operation on $\calP(\R)$. A measure $\mu\in \calP(\R)$ is said to be {\it $\circ$-selfdecomposable} (denoted $\mu\in L(\circ)$) if for any $c\in (0,1)$ there is $\mu_c\in \calP(\R)$ such that $\mu=D_c(\mu)\circ \mu_c$. In particular, if $\circ =\uplus$ (resp. $\circ=\boxplus$), then $\mu$ is said to be {\it Boolean selfdecomposable} (resp. {\it freely selfdecomposable}). 
\end{defn}

By Lemma \ref{lem:BFBP} \eqref{BFBP:hom}--\eqref{BFBP:dilation} and \cite[Theorem 4.6, Proposition 4.7]{BNT02}, we have 
\begin{align}\label{eq:BP_Lclasses}
\Lambda_B(L(\uplus))=L(\boxplus) \qquad \text{ and } \qquad L(\uplus)=\Lambda_B^{-1}(L(\boxplus)).
\end{align}
It is known that the free selfdecomposability has the following characterization  (see e.g. \cite[Subsection 2.2]{HT2016}).
\begin{lem}\label{prop:FSD}
A probability measure $\mu$ on $\R$ belongs to the class $L(\boxplus)$ if and only if $\mu$ is a freely infinitely divisible distribution of which free L\'{e}vy measure $\nu$ is Lebesgue absolutely continuous and the function $\R\setminus\{0\}\rightarrow [0,\infty)$, $x \mapsto |x|\frac{d\nu}{dx}(x)$ has a version with respect to the Lebesgue measure that is unimodal with mode $0$, i.e.\ non-decreasing on $(-\infty,0)$ and non-increasing on $(0,\infty)$. 
 \end{lem}
By the definition of $\Lambda_B$, we obtain a characterization for the Boolean selfdecomposability exactly in the same way as Lemma \ref{prop:FSD}, in which the free L\'evy measure is to be replaced with the Boolean L\'evy measure. 


\begin{prop} \label{prop:SD}
A probability measure $\mu$ on $\R$ belongs to class $L(\uplus)$ if and only if its Boolean L\'{e}vy measure is Lebesgue absolutely continuous and the function $k_\mu$ in  \eqref{boolean Levy measure} has a version with respect to the Lebesgue measure that is unimodal with mode $0$.  
\end{prop}
When we consider $\mu \in L(\uplus)$, for simplicity we always take $k_\mu$ itself to be unimodal with mode 0 unless specified otherwise.

Given a probability measure $\mu$, a typical method for proving it to be Boolean selfdecomposable is Proposition \ref{prop:SD}. To check that the function $k_\mu$ is unimodal with mode 0, we can calculate $k_\mu$ from Proposition \ref{kdensity}.  To check that the Boolean L\'evy measure is Lebesgue absolutely continuous, we provide a practical sufficient condition in Proposition \ref{kdensity2}. 

\begin{prop}\label{kdensity}
Suppose that $\mu\in \calP(\R)$. 
The function $k_\mu\colon\R\setminus\{0\}\rightarrow [0,\infty)$ defined by \eqref{boolean Levy measure}  is then given by
\begin{align}\label{prop:BSF}
k_\mu(x)= \lim_{\epsilon \to 0^+}  \frac{1}{\pi |x|} \Im [ F_\mu(x+i\epsilon) ], \qquad \text{a.e.\ } x\in \R\setminus\{0\}
\end{align}
and the Boolean Gaussian component of $\mu$ is given by $-\lim_{\epsilon \to 0^+} i \epsilon F_\mu(i\epsilon)$. 
\end{prop}
\begin{proof}
Both formulas are basic. Let $\nu$ be the Boolean L\'evy measure for $\mu$ and $(b,\tau)$ be the Boolean generating pair for $\mu$. By Proposition \ref{lem:K} and  relation \eqref{relation},  the F-transform of $\mu$ has the form
\begin{align}
F_\mu(z)&=z-b-\int_\R \left( \frac{1+t^2}{z-t}+t\right) \tau({\rm d}t)  \notag \\
&=z-b- \frac{a}{z}+\int_{\R} \left( \frac{1}{t-z}- \frac{t}{1+t^2}\right) |t|^2 \nu({\rm d}t).  \label{eq:F}  
\end{align}
Since the Lebesgue absolutely continuous part of $|t|^2 \nu({\rm d}t) $ is $|t|k_\mu(t) \,{\rm d}t$, the desired formula for $k_\mu$ is a consequence of the Stieltjes inversion formula (see e.g.\ \cite[Corollary 1.103]{HO07} or \cite[Theorem F.6]{Sch12}). The formula for Boolean Gaussian component follows from Dominated Convergence Theorem. 
\end{proof}

\begin{prop}\label{kdensity2}
Let $\mu\in \calP(\R)$. Suppose that there exists at most countable subset $C$ of $\R$ such that 
\begin{enumerate}[label=\rm(\arabic*)]
\item $0 \in C$,  
\item\label{item:A2} $
\lim_{\epsilon \to 0^+}  \Im [ F_\mu(x+i\epsilon)] \in [0,\infty)  
$
exists for all $x \in \R \setminus C$, 
\item\label{item:A3} $
\lim_{\epsilon \to 0^+} \epsilon F_\mu(x+i\epsilon) =0 
$
holds for all $x\in C\setminus\{0\}$. 
\end{enumerate}
Then the Boolean L\'{e}vy measure of $\mu$ is Lebesgue absolutely continuous. 
\end{prop}
\begin{proof} 
 Let $\nu$ be the Boolean L\'evy measure of $\mu$ and $\rho({\rm d}t) := |t|^2 \nu({\rm d}t)$. Note that $\rho(\{0\})=0$. It suffices to show that $\rho$ is Lebesgue absolutely continuous. 
 Let $\rho=\rho_{\rm sing} + \rho_{\rm ac}$ be the Lebesgue decomposition of $\rho$ into the singular part and absolutely continuous part with respect to the Lebesgue measure. 
From formula  \eqref{eq:F} and \cite[Theorem F.6]{Sch12},  $\rho_{\rm sing}$ is supported on the set 
\[
S_\rho^{\rm sing}:=\left\{ x \in \R: \lim_{\epsilon\to0^+}\Im[F_\mu(x+i\epsilon)] =\infty \right\}  
\]
in the sense that $\rho_{\rm sing}(\R\setminus S_\rho^{\rm sing})=0$. By assumption \ref{item:A2}, the set $S_\rho^{\rm sing}$ is contained in $C$ and hence is at most countable. This implies that $\rho_{\rm sing}$ is a discrete measure. Then assumption \ref{item:A3} implies that $\rho$ does not have an atom in $C$, so that $\rho_{\rm sing}=0$.  
\end{proof}

\begin{rem} It is clear that $L(\uplus) \nsubseteq L(\boxplus)$ since the Bernoulli distribution $\frac{1}{2}(\delta_1+\delta_{-1})$ is in $L(\uplus)$, but not in $L(\boxplus)$. 

The following example shows that $L(\boxplus) \nsubseteq L(\uplus)$. Let $\FS$ be the positive free $1/2$--stable distribution defined by 
\[
\varphi_\FS(z) = \sqrt{-z}, \qquad z \in \C^+,
\]
where the square root above is defined as the principal branch (see \cite[Appendix]{BP} for details). Stability implies selfdecomposability, so that $\FS \in L(\boxplus)$.  To see $\FS \notin L(\uplus)$,  inverting $F_\FS^{-1}(z) = z + \sqrt{-z}$ yields 
\[
F_\FS(z) = \frac{2z - 1 -\sqrt{1-4z}}{2}, \qquad z\in \C^+.  
\]
By Proposition \ref{kdensity2}, the Boolean L\'evy measure is Lebesgue absolutely continuous. By Proposition \ref{kdensity} the function $k_{\FS}$ can be calculated into 
\[
k_{\FS}(x) =   \frac{\sqrt{4x-1}}{2\pi x} \cdot \mathbf{1}_{[1/4, \infty)}(x).  
\]
Because $k_{\FS}$ is not non-increasing on $(0,\infty)$, the desired conclusion $\FS \notin L(\uplus)$ follows. 
\end{rem}

\begin{ex}\label{ex:twopointmeasure_BSD}
The probability measure having the Boolean L\'evy triplet $(a,0,\gamma)$ with $a>0$ is called the \emph{Boolean Gaussian distribution} and will be denoted $B(\gamma,a)$. It has mean $\gamma$, variance $a$ and has the form 
\begin{align}\label{eq:explictform2}
\left(\frac{1}{2}-\frac{\gamma}{2\sqrt{\gamma^2+4a^2}} \right)\delta_{(\gamma-\sqrt{\gamma^2+4a^2})/2} +\left(\frac{1}{2}+\frac{\gamma}{2\sqrt{\gamma^2+4a^2}} \right)\delta_{(\gamma+\sqrt{\gamma^2+4a^2})/2}.
\end{align}
Because the Boolean L\'evy measure is zero, $B(\gamma,a)$ is Boolean selfdecomposable.  
Note that $B(a,\gamma)$ can be written in the simpler form 
\[
\frac{1}{\beta-\alpha} ( \beta \delta_\beta-\alpha\delta_\alpha) 
\]
with parameters $\alpha=(\gamma-\sqrt{\gamma^2+4a^2})/2<0$ and $\beta=(\gamma+\sqrt{\gamma^2+4a^2})/2>0$. In fact, the map $(\gamma,a) \mapsto (\alpha,\beta)$ gives a bijection from $\R \times (0,\infty)$ to $(-\infty,0) \times (0,\infty)$. 

In particular, one can see that $p\delta_1+ (1-p)\delta_{-1}$ for $p\in (0,1)$ is Boolean Gaussian distribution if and only if $p=1/2$.

\end{ex}

\begin{ex} 
The probability measure $\bb_{\alpha, \rho}$ characterized by  
\[
\eta_{\bb_{\alpha,\rho}}(z)= - (e^{i \rho\pi}z)^{\alpha}, \qquad z\in \C^- 
\]
is called a Boolean (strictly) stable distribution, where the parameter $(\alpha,\rho)$ belongs to the set 
\begin{equation*}
\mathfrak{A}=\{(\alpha,\rho): \alpha \in(0,1], \rho \in [0,1]\} \cup \{(\alpha,\rho): \alpha\in(1,2], \rho \in [1-\alpha^{-1}, \alpha^{-1}]\}.
\end{equation*}
The Boolean stable distributions satisfy the relation $\bb_{\alpha,\rho} = D_c(\bb_{\alpha,\rho}) \uplus D_{(1-c^\alpha)^{1/\alpha}}(\bb_{\alpha,\rho})$ for any $c \in (0,1)$, and hence are Boolean selfdecomposable.  Proposition \ref{kdensity} implies the formula 
\[
k_{\bb_{\alpha,\rho}}(x)= \frac{\sin (\alpha \rho \pi)}{\pi} x^{-\alpha} \mathbf1_{(0,\infty)}(x) + \frac{\sin (\alpha (1-\rho) \pi)}{\pi} |x|^{-\alpha} \mathbf1_{(-\infty,0)}(x). 
\]
The special case $\bb_{2,1/2}$ coincides with the Bernoulli distribution $\bern$. For further information on $\bb_{\alpha, \rho}$, see e.g.\ \cite{SW97,HS15}.  
\end{ex}

\subsection{Regularity for Boolean selfdecomposable distributions}

A general regularity result on boolean selfdecomposable distributions is established as follows. Note that we exclude the well understood Boolean Gaussian distributions (see Example \ref{ex:twopointmeasure_BSD}) from the statement in order to have a non-empty support for the function $k_\mu$.  

\begin{thm} \label{thm:Lebesgue_decomposition}
Let $\mu \in L(\uplus)$ with the function $k_\mu$ not identically equal to zero. Let $\alpha := \inf\{x\in \R: k_\mu(x)\ne0\} \in [-\infty,0]$ and  $\beta := \sup\{x\in \R: k_\mu(x)\ne0\} \in [0,\infty]$. 
Then 
\begin{enumerate}
\item\label{item:AC} $\mu|_{(\alpha,\beta)}$ is absolutely continuous with respect to the Lebesgue measure.  
\item\label{item:atom}  $\mu|_{(-\infty,\alpha]}$ is the zero measure or a delta measure if $\alpha > -\infty$. The same holds for $\mu|_{[\beta,\infty)}$ if $\beta < \infty$. 
\item\label{item:atom0} $\mu(\{0\})=0$. 
\end{enumerate}
In particular, $\mu$ has no singular continuous part with respect to the Lebesgue measure and has at most two atoms. 
\end{thm}

\begin{proof}  
Let us take the Pick-Nevanlinna representation of $F_\mu$ in the form
\begin{equation}\label{eq:Pick}
F_\mu(z) =  z  -b-\frac{a}{z}+  \int_{\alpha}^\beta \left( \frac{1}{t-z} - \frac{t}{1+t^2} \right) |t| k_\mu(t) \,{\rm d}t,   
\end{equation}
where $a \ge0, b \in \R$; see \eqref{eq:F}.

\vspace{2mm}\noindent 
\eqref{item:atom} It suffices to prove the statement for $(-\infty,\alpha]$ by symmetry. We assume $\alpha > -\infty$. The formula \eqref{eq:Pick} enables $F_\mu$ to be extended analytically to $\C^+ \cup (-\infty,\alpha)$. The extended function, still denoted $F_\mu$, takes real values on $(-\infty,\alpha)$; hence, the value of the function $\Im[G_\mu(z)] = -\Im[F_\mu(z)]/ |F_\mu(z)|^2$ is zero on $(-\infty, \alpha)\setminus \{\text{zeros of $F_\mu|_{(-\infty, \alpha)}$}\}$.  By the Stieltjes inversion formula, we conclude that 
\[
\mu\left((-\infty, \alpha)\setminus \{\text{zeros of $F_\mu|_{(-\infty, \alpha)} $}\}\right) =0. 
\]
Because  $F_\mu' >1$ on $(-\infty, \alpha)$, $F_\mu$ has at most one zero on  $(-\infty, \alpha)$.  

If $F_\mu$ has no zeros on $(-\infty, \alpha)$, then $\mu|_{(-\infty,\alpha]}$ is the zero measure or a delta measure at $\alpha$. 

If $F_\mu$ has a zero $x_0$ in $(-\infty, \alpha)$, then $\mu (\{x_0\}) = \lim_{z \to x_0} \frac{z-x_0}{F_\mu(z)} = \frac{1}{F_\mu'(x_0)} >0$. In this case we need to check that $\mu(\{\alpha\})=0$. 
Because $F_\mu$ is strictly increasing and $F_\mu(x_0)=0$, we have $\lim_{x\to \alpha-0}F_\mu(x) \in (0,\infty]$ and hence $G_\mu(\alpha) :=\lim_{x\to \alpha-0}G_\mu(x) \in [0,\infty)$. It can be verified that $G_\mu(\alpha) = \lim_{\epsilon \to0^+} G_\mu(\alpha +i \epsilon)$ directly or by Lindel\"of's theorem \cite[Theorem 1.5.7]{BCD2020}. Therefore, $\mu(\{\alpha\}) =\lim_{\epsilon\to0^+} i\epsilon G_\mu(\alpha+ i\epsilon)=0$. 


\vspace{2mm}
\noindent 
\eqref{item:atom0} Because $k_\mu$ is unimodal and is not identically zero, either $\alpha<0$ or $\beta >0$ holds.  Without loss of generality we may work on the latter case. Then $\delta := \inf_{x \in (0,\beta/2)} k_\mu(x) >0$. 
Because $\mu(\{0\}) = \lim_{\epsilon \to0^+} i\epsilon / F_\mu(i\epsilon)$, the desired assertion $\mu(\{0\})=0$ follows from $\lim_{\epsilon \to0^+} \Im[F_\mu(i\epsilon)]/\epsilon =\infty$. The latter is the case because 
\begin{align*}
\frac{\Im[F_\mu(i\epsilon)]}{\epsilon} 
= 1 +\frac{a}{\epsilon^2} + \int_{\alpha}^\beta \frac{|t| k_\mu(t)}{t^2+\epsilon^2} {\rm d}t 
\ge \delta\int_{0}^{\beta/2} \frac{ t }{t^2+\epsilon^2}\, {\rm d}t 
\to  \delta\int_{0}^{\beta/2} \frac{1}{t}\, {\rm d}t = \infty 
\end{align*}
as $\epsilon \to 0^+$. 

\vspace{2mm}\noindent 
\eqref{item:AC} Let $\mu=\mu_{\rm sing} + \mu_{\rm ac}$ be the Lebesgue decomposition of $\mu$ into the singular part and absolutely continuous part with respect to the Lebesgue measure.  From the established fact  \eqref{item:atom0}, it suffices to prove $\mu_{\rm sing}((\alpha, \beta)\setminus\{0\}) =0$. 
As in the proof of Proposition \ref{kdensity2},  $\mu_{\rm sing}(\R\setminus S)=0$, where
\[
S:=\left\{ x \in \R: \lim_{\epsilon\to0^+}\Im[- G_\mu(x+i\epsilon)] =\infty \right\}.   
\]

Let us assume that $\beta >0$; otherwise we need not work on the interval $(0,\beta)$. 
Because $k_\mu$ is unimodal with mode 0, for every $0<\beta' <\beta$ the number $\delta':= \inf_{x \in (0,\beta']} k _\mu(x) $ is positive. For every $\epsilon >0$ and $x \in (0,\beta')$ we have 
\begin{align*}
\Im[F_\mu(x+ i\epsilon)] 
&= \epsilon  +\frac{a}{\epsilon} + \int_{\alpha}^\beta \frac{\epsilon |t| k_\mu(t)}{(t-x)^2+\epsilon^2} {\rm d}t \\
&\ge  \delta' x\int_{x}^{\beta'}  \frac{\epsilon}{(t-x)^2+\epsilon^2} {\rm d}t \\
&=  \delta' x\arctan \frac{\beta' -x}{\epsilon}, 
\end{align*}
which tends to $\delta' x  \pi/2$ as $\epsilon$ tends to $0$. 
This implies that 
\[
\limsup_{\epsilon\to0^+} \Im[- G_\mu(x+i\epsilon)] =\limsup_{\epsilon\to0^+} \frac{\Im[F_\mu(x+i\epsilon)]} {|F_\mu(x+i\epsilon)|^2}  \le \limsup_{\epsilon\to0^+} \frac{1} {\Im[F_\mu(x+i\epsilon)]}  \le \frac{2}{\delta' x  \pi} < \infty
\]
for all $x \in (0,\beta')$. We can thus conclude that $(0,\beta) \subseteq \R\setminus S$, i.e.\  $\mu_{\rm sing}((0,\beta)) =0$. Similarly or by symmetry, we get  $\mu_{\rm sing}((\alpha, 0)) =0$ if $\alpha<0$. 
\end{proof}

\begin{rem}
The delta measures and Boolean Gaussian distributions are obviously singular distributions. Oher Boolean selfdecomposable distributions may also have a non-zero singular part, see e.g.\ Example \ref{ex:SD} \eqref{item:Kesten}, \eqref{item:B}. By contrast, all classical selfdecomposable distributions and free ones except for the delta measures are Lebesgue absolutely continuous (see \cite[Theorem 27.13]{Sato} and \cite[Remark 2]{HT2016}, respectively). 
\end{rem}

\begin{rem} \label{rem:atom}
In the setting of Theorem \ref{thm:Lebesgue_decomposition}, the points $\alpha$ and $\beta$ may or may not be an atom of $\mu$. To see this let $\mu$ be the Boolean selfdecomposable probability measure defined by 
\[
F_{\mu}(z) = z -b + \int_0^1 \frac1{t-z} |t| \frac{(1-t)^p}{|t|} \,{\rm d}t, \qquad z \in \C^+,  
\]
where $p> 0$ is a parameter and $b \in \R$ is defined so that $F_{\mu}(1):=  1 - b - \int_0^1(1-t)^{p-1} \,{\rm d}t=0$. One sees that 
\begin{align*}
\frac{F_\mu(z)}{z-1}
&= \frac{F_\mu(z) - F_\mu(1)}{z-1} = 1 + \int_0^1 \frac{(1-t)^{p-1}}{z-t}   \,{\rm d}t  \to q:= 1 + \int_0^1 (1-t)^{p-2} \,{\rm d}t \in (1,\infty]
\end{align*}
 as $z \to 1$ satisfying $\Re[z] \ge1$. This implies that $\mu(\{1\}) = 1/q$ which is positive if $p>1$ and zero if $0 < p \le 1$.  
 
\end{rem}

\begin{ex}[Mixture of Cauchy distribution and $\delta_0$] 
Let $\kappa_p = \frac{p}{\pi(1+x^2)} \mathbf{1}_{\R}(x)\,{\rm d}x + (1-p)\delta_0$ for $p \in [0,1]$. Its reciprocal Cauchy transform is given by 
\[
F_{\kappa_p}(z) =\frac1{\frac{p}{z+i} + \frac{1-p}{z}} = \frac{z (z+i)}{z + (1-p)i}, \qquad z \in \C^+.
\]
By Proposition \ref{kdensity2}, the Boolean L\'evy measure is Lebesgue absolutely continuous. By Proposition \ref{kdensity}, the function $k_{\kappa_p}$ is given by 
\[
k_{\kappa_p}(x) = \frac{1}{\pi}\cdot\frac{p|x|}{x^2 + (1-p)^2}\cdot \mathbf{1}_{\R\setminus\{0\}}(x). 
\]
As long as $p\in (0,1)$, this is not unimodal with mode 0 and hence $\kappa_p$ is not Boolean selfdecomposable (cf.\ Theorem \ref{thm:Lebesgue_decomposition} (3)).   
\end{ex}

According to Theorem \ref{thm:Lebesgue_decomposition}, every Boolean selfdecomposable distribution $\mu$ with $k_\mu\neq 0$ has at most two atoms. Here we completely determine the  two point measures which are Boolean selfdecomposable. 

\begin{prop}\label{prop:twopointmeasure_BSD}
A two point probability measure on $\R$ belongs to $L(\uplus)$ if and only if it is a Boolean Gaussian distribution.
\end{prop}
\begin{proof}
The Boolean Gaussian distribution $B(\gamma,a)$ is a Boolean selfdecomposable two point probability measure for all $\gamma\in \R$ and $a>0$, see \eqref{eq:explictform2}. 

Conversely, we assume that a two point probability measure $\mu$ belongs to $L(\uplus)$. By Proposition \ref{kdensity} we get $k_\mu(x)=0$ for a.e.\ $x\in \R\setminus\{0\}$ because $F_\mu$ is of the form $(z^2 + p z +q)/(z-r)$, $p,q,r \in \R$. Therefore $\mu$ has a Boolean L\'evy triplet $(a,0,\gamma)$ for some $\gamma \in \R$ and some $a\ge0$. The Boolean Gaussian component $a$ is nonzero; otherwise $\mu$ would be a Dirac measure. Consequently, we have $\mu=B(\gamma,a)$ for $\gamma \in \R$ and $a>0$.
\end{proof}

\subsection{Boolean selfdecomposability of shifted probability measures}

For any $a \in \R$, ``the Boolean shift'' $\mu \mapsto \mu \uplus \delta_a$ preserves the class $L(\uplus)$. On the other hand, the usual shift $\mu \mapsto \mu \ast \delta_a$ does not preserve $L(\uplus)$, which can be observed from Proposition  \ref{prop:twopointmeasure_BSD}. This phenomenon is investigated in details below. The function $\ell_\lambda$ defined in \eqref{boolean Levy measure} plays a key role.

\begin{lem}\label{lem:shift} Suppose that $\lambda \in \calP(\R)$ satisfies the condition 
\begin{enumerate}[label=\rm(C)]
\item\label{item:C} the Boolean Gaussian component  is zero and the Boolean L\'evy measure is  Lebesgue absolutely continuous. 
\end{enumerate}
Then for any $m \in \R$ the measure $\lambda \ast \delta_m$ also satisfies condition \ref{item:C} and 
\begin{equation}\label{eq:spectral_shift}
\ell_{\lambda\ast \delta_m}(t) = \ell_\lambda(t-m) \quad \text{for a.e.~}  t \in \R. 
\end{equation}
\end{lem}
\begin{proof}
 According to \eqref{boolean Levy measure}, formula \eqref{eq:F} can be expressed in the form
 \begin{align}
F_{\lambda}(z) &=z - b-\frac{a}{z} + \int_{\R\setminus\{0\}} \left( \frac{1}{t-z}- \frac{t}{1+t^2}\right) \ell_\lambda(t){\,\rm d}t,   \label{eq:F3}
\end{align}
 where $a$ is the Boolean Gaussian component, which is now 0. 
This yields 
\begin{align}
F_{\lambda\ast \delta_m}(z) =F_\lambda(z-m) =z- b_m + \int_{\R\setminus\{0\}} \left( \frac{1}{t-z}- \frac{t}{1+t^2}\right) \ell_\lambda(t-m) {\,\rm d}t  \label{eq:F2}
\end{align}
for some $b_m \in \R$. We can therefore conclude that the Boolean Gaussian component  of $\lambda\ast \delta_m$ is zero and its Boolean L\'evy measure equals $t^{-2} \ell_\lambda(t-m) \mathbf{1}_{\R\setminus\{0\}}(t)\, {\rm d}t$, as desired. 
\end{proof}

For a clear statement, we define $\calQ$ to be the set of probability measures $\lambda$ on $\R$ such that 

\begin{enumerate}[label=\rm(H\arabic*)]
\item\label{item:H1} the Boolean Gaussian component of $\lambda$ is zero, 

\item\label{item:H2} the Boolean L\'evy measure of $\lambda$ admits the form $t^{-2}\ell_\lambda (t)\mathbf{1}_{\R\setminus\{0\}}(t)\,{\rm d}t$, where $\ell_\lambda\colon \R \to [0,\infty)$ is right continuous on $\R\setminus \{m_\lambda\}$ for some $m_\lambda \in \R$. 
\end{enumerate}

\begin{thm} \label{thm:shift}

\begin{enumerate}
\item \label{item:Gaussian} If $\lambda \in \calP(\R) \setminus  \calQ$  then $\lambda \ast \delta_m\notin L(\uplus)$ for any $m\in \R \setminus\{0\}$. 

\item \label{item:plus} If $\lambda\in \calQ$ and there exist  $a,b \in \R\setminus \{m_\lambda\}$ with $a < b$ and $\ell_\lambda(a) < \ell_\lambda(b)$ then $\lambda \ast \delta_m \notin L(\uplus)$ for any $m > \frac{b\ell_\lambda(a) - a \ell_\lambda(b)}{\ell_\lambda(b) -\ell_\lambda(a)}$. 

\item \label{item:minus} If $\lambda\in \calQ$ and there exist $c,d \in \R\setminus \{m_\lambda\}$ with $c<d$ and $\ell_\lambda(c) >\ell_\lambda(d)$ then $\lambda \ast \delta_{m} \notin L(\uplus)$ for any $m  <   \frac{c\ell_\lambda(d) - d \ell_\lambda(c)}{\ell_\lambda(c) -\ell_\lambda(d)}$. 

\item  \label{item:flat} If $\lambda\in \calQ$ and $\ell_\lambda$ is constant on $\R\setminus \{m_\lambda\}$ then $\lambda \ast \delta_m \in L(\uplus)$ for all $m \in \R$. 
\end{enumerate}
\end{thm}


\begin{proof}
\eqref{item:Gaussian}  Suppose first that $\lambda$ does not satisfy \ref{item:H1}, i.e.\ it has a positive Boolean Gaussian component $a$.  
Similarly to \eqref{eq:F2}, we obtain 
\begin{align}
F_{\lambda\ast \delta_m}(z) &=z-b_m + \frac{a}{m-z}+\int_{\R\setminus\{0\}} \left( \frac{1}{t-z}- \frac{t}{1+t^2}\right) (t-m)^2 {\rm d}\nu(t-m)  \label{eq:F5}
\end{align}
for some $b_m \in \R$, where $\nu$ is the Boolean L\'evy measure of $\lambda$.  
This implies that the Boolean L\'evy measure of $\lambda \ast \delta_m$ is given by $a m^{-2}\delta_m + t^{-2}(t-m)^2{\rm d}\nu(t-m)$ as soon as $m\ne0$. Because this is not Lebesgue absolutely continuous, $\lambda\ast \delta_m$ is not Boolean selfdecomposable. 

To complete the proof of \eqref{item:Gaussian}, it suffices to prove that if $\lambda$ satisfies \ref{item:H1} and $\lambda \ast \delta_{m_0} \in L(\uplus)$ for some $m_0 \in \R$ then $\lambda$ satisfies \ref{item:H2}.   
According to Lemma \ref{lem:shift}, $\lambda$ has a Lebesgue absolutely continuous Boolean L\'evy measure and 
\begin{equation}\label{eq:spectral_shift}
\ell_\lambda(t)= \ell_{(\lambda \ast \delta_{m_0}) \ast \delta_{-m_0}}(t) = \ell_{\lambda \ast \delta_{m_0}}(t+m_0) = |t+m_0| k_{\lambda\ast \delta_{m_0}}(t+m_0), \qquad \text{a.e.~} t \in \R \setminus\{0\}. 
\end{equation}
Because $k_{\lambda\ast \delta_{m_0}}$ is unimodal with mode 0, it has a version that is right continuous on $\R\setminus\{0\}$. Therefore, $\ell_\lambda$ has a version that is right continuous on $\R\setminus\{-m_0\}$, i.e.\ \ref{item:H2} holds.

\vspace{2mm}\noindent
\eqref{item:plus} For all $m > M:= \frac{b\ell_\lambda(a) - a \ell_\lambda(b)}{\ell_\lambda(b) -\ell_\lambda(a)}$, observe that $0< m+a < m+b$ and, by Lemma \ref{lem:shift}, 
\begin{align}
k_{\lambda\ast \delta_m}(m+a) - k_{\lambda\ast \delta_m}(m+b) = \frac{\ell_\lambda(a)}{m+a} - \frac{\ell_\lambda(b)}{m+b}  <0.   \label{eq:shift_k}  
\end{align}
 Because $\ell_\lambda$ is right continuous at $a$ and $b$, the inequality \eqref{eq:shift_k} still holds when $a, b$ are replaced with $a+\epsilon, b+\epsilon$ for sufficiently small $\epsilon>0$, respectiely; hence, for any $m>M$,  $k_{\lambda\ast \delta_m}$ does not have a version that is non-increasing on  $(0,\infty)$, and consequently  $\lambda\ast\delta_m$ is not in $L(\uplus)$. The assertion \eqref{item:minus} is similar. 

\vspace{2mm}\noindent
\eqref{item:flat} The distribution $\lambda\ast\delta_m$ is then a delta measure (when $\ell_\lambda =0$) or Cauchy distribution (when $\ell_\lambda$ is a positive constant function) for any $m \in \R$, which is Boolean selfdecomposable.  
\end{proof}

\begin{rem} In case \eqref{item:plus} it might happen that $\lambda \ast \delta_m$ is in $L(\uplus)$ for all sufficiently small $m$. To detect such an example, the function $\ell_\lambda$ needs to be non-decreasing on $\R$; otherwise case \eqref{item:minus} would apply. Let $\lambda$ be a probability measure with Boolean L\'evy triplet $(0, k_\lambda(t){\,\rm d}t, 0)$, where
\[
k_\lambda(t)= \frac{1}{|t|}\left[ \arctan t + \frac{\pi}{2} \right] ,\qquad t \in \R\setminus\{0\}.  
\]
 It can be checked by calculus that the function 
\[
k_{\lambda \ast \delta_m}(t) := \frac{1}{|t|} \left[ \arctan (t-m) + \frac{\pi}{2} \right] ,\qquad t \in \R\setminus\{0\} 
\]
is unimodal with mode $0$ if and only if $m \le \pi/2$, and hence $\lambda \ast \delta_m \in L(\uplus)$ for any such $m$.  
\end{rem}

\subsection{Several examples of Boolean selfdecomposable distributions}

 We observe distributional properties of Boolean selfdecomposable distributions through several examples. 

\begin{ex}\label{ex:SD}
  We give several examples in the class $L(\uplus)$ as follows.
\begin{enumerate} 
\item The free Poisson (or Marchenko-Pastur) distribution $\MP_\lambda$ is a probability measure defined by
\begin{equation*}
\max\{1-\lambda,0\} \delta_0 + \frac{1}{2\pi x}\sqrt{((1+\sqrt{\lambda})^2-x)(x-(1-\sqrt{\lambda})^2)} \cdot \,\mathbf{1}_ {((1-\sqrt{\lambda})^2,(1+\sqrt{\lambda})^2)}(x){\,\rm d}x, \qquad \lambda>0. 
\end{equation*}
Its F-transform is given by 
\[
F_{\MP_\lambda}(z)=\frac{z+1-\lambda+\sqrt{(z-(1-\sqrt{\lambda})^2)(z-(1+\sqrt{\lambda})^2)}}{2}, \qquad z\in \C^+,
\]
where the square root $\sqrt{w}$ is defined continuously on angles $\arg w \in(0,2\pi)$. 
By Propositions \ref{kdensity2} and \ref{kdensity}, the Boolean L\'evy measure is Lebesgue absolutely continuous and the function $k_{\MP_\lambda}$ is given by
\[
k_{\MP_\lambda}(x)=\frac{\sqrt{(x-(1-\sqrt{\lambda})^2)((1+\sqrt{\lambda})^2-x)} }{2\pi x}\cdot \mathbf{1}_{( (1-\sqrt{\lambda})^2, (1+\sqrt{\lambda})^2  )} (x),
\]
and therefore $\MP_\lambda \in L(\uplus)$ if and only if $\lambda=1$. Consequently, $L(\uplus)$ is not closed under the free convolution since for example $\MP_1\boxplus \MP_1=\MP_2\notin L(\uplus)$.

\item Let $S(m,\sigma^2)$ be the Wigner's semicircle law, that is, its probability density function is defined by
$$
\frac{1}{2\pi \sigma^2} \sqrt{4\sigma^2-(x-m)^2} \cdot \mathbf{1}_{[m-2\sigma,m+2\sigma]}(x), \qquad m\in\R, \hspace{2mm} \sigma>0.
$$
Its F-transform is known to be
$$
F_{S(m,\sigma^2)}(z)=\frac{z-m+\sqrt{(z-m)^2-4\sigma^2}}{2}, \qquad z\in \C^+.
$$
By Proposition \ref{kdensity2} the Boolean L\'evy measure is Lebesgue absolutely continuous and by Proposition \ref{kdensity} we have 
$$
k_{S(m,\sigma^2)}(x)=\frac{\sqrt{4\sigma^2-(x-m)^2}}{2\pi |x|}\cdot \mathbf{1}_{[m-2\sigma,m+2\sigma] \setminus\{0\} }(x),
$$
and therefore $S(m,\sigma^2)\in  L(\uplus)$ if and only if $m-2\sigma \le 0 \le m+2\sigma$.


\item\label{item:Kesten} The measure $\bern^{\boxplus t}$ is given by
\[
\max\left\{\frac{2-t}{2},0\right\} (\delta_{t}+\delta_{-t})+ \frac{1}{2\pi} \cdot \frac{t\sqrt{4(t-1)-x^2}}{t^2-x^2}\cdot \mathbf{1}_{[-2\sqrt{t-1},2\sqrt{t-1}]}(x){\,\rm d}x, \qquad t>1.
\]
It is called the Kesten distribution. The F-transform of $\bern^{\boxplus t}$ is given by
\[
F_{\bern^{\boxplus t}}(z)=\frac{(t-2)z+t\sqrt{z^2-4(t-1)}}{2(t-1)}, \qquad z\in \C^+.
\]
By Propositions \ref{kdensity2} and \ref{kdensity}, the Boolean L\'evy measure is Lebesgue absolutely continuous and 
\[
k_{\bern^{\boxplus t}}(x)=\frac{t\sqrt{4(t-1)-x^2}}{2\pi (t-1)|x|}\cdot \mathbf{1}_{[-2\sqrt{t-1},2\sqrt{t-1} ]\setminus\{0\}}(x).
\]
Because $k_{\bern^{\boxplus t}}$ is unimodal with mode $0$ for all $t>1$,  we conclude that $\bern^{\boxplus t} \in L(\uplus)$ for all $t>1$. 

\item\label{item:B} Let $\mu(p,r)$ be the 2-parameter Fuss-Catalan distribution. Recall that $\mu(p,r)\in I(\boxplus)$ if and only if either $0<r\le \min\{p/2,p-1\}$ or $1\le p=r\le 2$ (see \cite[Theorem 4.1]{MSU20}). Then we can define $\widetilde{\mu}(p,r):= \Lambda_B^{-1} (\mu(p,r))$ for such pairs $(p,r)$. Since $\mu(p,r) \in  L(\boxplus)$ if and only if $1\le p=r\le 2$ (see \cite[Theorem 4.3]{MSU20}) and \eqref{eq:BP_Lclasses}, we have $\widetilde{\mu}(p,p)\in  L(\uplus)$ for $1\le p\le2$. In particular, $\widetilde{\mu}(2,2)=B(2,1)$ is a Boolean Gaussian distribution. By Lemma \ref{lem:BFBP} (i) and \cite[Proposition 3.6]{MSU20}, for $1 \le p <2$ the Boolean L\'evy triplet of $\widetilde{\mu}(p,p)$ is given by $(0,  |x|^{-1}k_{\mu(p,p)}(x)\,{\rm d}x, p)$, 
where
\begin{align*}
k_{\mu(p,p)}(x) =- \frac{\sin(p \pi)}{\pi}\left(\frac{1+x}{-x}\right)^{p}\cdot \mathbf{1}_{(-1,0)}(x).
\end{align*}

From computations similar to \cite[Proposition 3.6]{MSU20}, we obtain
\begin{align*}
\eta_{\widetilde{\mu}(p,p)}(z) = (1 + z)^{p}-1, \qquad z\in \C^-.
\end{align*}
Then its Cauchy transform is as follows:
\begin{align*} 
G_{\widetilde{\mu}(p,p)}(z) = \frac{1}{z[2-(1 + z^{-1})^{p}]} ,\qquad z \in \C^{+}.
\end{align*}
By the Stieltjes inversion formula, the measure $\widetilde{\mu}(p,p)$ for $1 \le p <2$ is given by the following explicit form:
\begin{align*}
\frac1{p(2-2^{1-1/p})}\delta_{(2^{1/p}-1)^{-1}}
+\frac{\sin(p\pi)\left|\frac{1+x}{x}\right|^{p}}{\pi x(4 - 4\cos(p\pi) \left|\frac{1+x}{x}\right|^{p} + 
      \left|\frac{1+x}{x}\right|^{2p})}\mathbf{1}_{(-1,0)}(x)\, \dd x.
\end{align*}
According to the example, the discrete parts of Boolean selfdecomposable distributions come from not only Boolean Gaussian component but also its Boolean L{\'e}vy measure. 
\end{enumerate}
 \end{ex}

\section{Boolean selfdecomposability for normal distributions}

In this section, we investigate Boolean selfdecomposability for normal distributions $N(m,v)$. The question is unimodality of $k_{N(m,v)}$ with mode $0$. It turns out that $k_{N(m,v)}$ has an explicit formula. To compute it, we make use of  the following differential equation (see \cite[(8.1.6)]{Kerov98}) which can be derived by integration by parts. 

\begin{lem}\label{lem:ODE}
The following differential equation holds:
\[
G_{\ND}'(z)=1-zG_{\ND}(z), \qquad z\in \C^+.
\]
\end{lem}


Using the above lemma, we can explicitly compute the functions $G_{N(0,1)}$ and $k_{\ND}$. 

\begin{prop}\label{prop:k} Let $h(z) = \int_0^z e^{t^2}\dd t, z \in \C$. Then 
\begin{equation}\label{eq:CauchyN}
G_{\ND}(z) = e^{- \frac{1}{2}z^2 } \left[ -i \sqrt{\frac{\pi}{2}}   + \sqrt{2}   h\left(\frac{z}{\sqrt{2}}\right) \right], \qquad z \in \C^+
\end{equation}
and the Boolean L\'evy triplet for $N(0,1)$ is $(0,k_{\ND}(x)\,{\rm d}x, 0)$, where 
\begin{align}\label{eq:normal_k}
k_{\ND}(x)=\frac{1}{\pi}\sqrt{\frac{2}{\pi}} \cdot \frac{1}{|x|e^{-\frac{x^2}{2}} \left[ 1 + \frac{4}{\pi} h\left(\frac{x}{\sqrt{2}}\right)^2\right] } \mathbf{1}_{\R\setminus\{0\}}(x).  
\end{align}
\end{prop}

\begin{proof} 
The standard variation of constants method for the differential equation in Lemma \ref{lem:ODE} yields the solution
\begin{equation*}
G_{\ND}(z) = e^{- \frac{1}{2}z^2 } \left[ C  + \int_0^z e^{\frac1{2}t^2} \,{\rm d}t\right], \qquad z \in \C^+ 
\end{equation*}
for some $C \in \C$. Because $N(0,1)$ is symmetric about 0, $\Re[G_{N(0,1)}(i y)] =0$ for all $y >0$ and so $\Re[C]=0$. In order that the density of $N(0,1)$ is obtained as the limit $(-1/\pi)\lim_{\epsilon\to0^+} G_{\ND}(x+i\epsilon)$, we must have $\Im [C] = - \sqrt{\frac{\pi}{2}}$, concluding the desired formula \eqref{eq:CauchyN}. Because $F_{N(0,1)}$ has a continuous extension to $\C^+\cup \R$, the Stieltjes inversion \eqref{eq:K2} implies that the Boolean Gaussian component is zero and the Boolean L\'evy measure is Lebesgue absolutely continuous.  
By Proposition \ref{kdensity} we have \eqref{eq:normal_k}. By \eqref{eq:K1} and $\Re[F_{N(0,1)}(i)]=0$ we obtain $b=0$. Finally, because $\tau({\rm d}x) = \frac{|x| k_{N(0,1)}(x)}{1+x^2}\, {\rm d}x$ is symmetric about zero, \eqref{relation} yields $\gamma=0$. 
\end{proof}

Analyzing the function $k_{N(0,1)}$ we are able to demonstrate the following. 
\begin{thm}\label{thm:Normal}
For all $v>0$, we have $N(0, v)\in L(\uplus)$.
\end{thm}

\begin{proof}
Since $D_{\sqrt{v}} (\ND)=N(0,v)$, it suffices to show that $\ND\in L(\uplus)$. In order to show $\ND \in L(\uplus)$, we prove unimodality for $k_{\ND}$ with mode $0$ using Proposition \ref{prop:k}. 
Because the function $k_{\ND}$ is symmetric with respect to $0$ it is enough to prove that $k_{\ND}(x)$ is non-increasing on $(0,\infty)$. 
For this it suffices to show that $\ell_{N(0,1)}$ is  non-increasing on $(0,\infty)$, i.e.\  the function 
\begin{equation}\label{eq:f}
f(x) = e^{-\frac{x^2}{2}} \left[ 1 + \frac{4}{\pi} h\left(\frac{x}{\sqrt{2}}\right)^2\right]
\end{equation}
is non-decreasing on $(0,\infty)$. For concise calculations, let us instead work on the rescaled function $g(x) = f(\sqrt{2}x) $. 

To begin we compute
\[
g'(x) =  -2x  e^{-x^2} \left[ 1 +\frac{4}{\pi} h(x)^2 \right] + \frac{8}{\pi} h(x) =  -\frac{8x e^{-x^2}}{\pi} \left[ h(x)^2 - \frac{e^{x^2}}{x} h(x) + \frac{\pi}{4}  \right]. 
\]
The desired inequality, $g' \ge0$ on $(0,\infty)$, is therefore equivalent to 
\begin{equation}\label{eqh}
\frac1{2}\left[\frac{e^{x^2}}{x} - \sqrt{\left(\frac{e^{x^2}}{x}\right)^2 - \pi }  \right] \le  h(x) \le \frac1{2}\left[\frac{e^{x^2}}{x} + \sqrt{\left(\frac{e^{x^2}}{x}\right)^2 - \pi }  \right]. 
\end{equation}
Note here that $(e^{x^2}/x)^2 \ge 2e >\pi$ for all $x >0$.

For the upper bound of \eqref{eqh}, we use the following supplementary inequality
\begin{equation}\label{eqh2}
h(x) < \frac{e^{x^2}-1}{x}, \qquad x>0. 
\end{equation}
This can be easily verified with calculus:  the function $H(x) := (e^{x^2}-1)/x- h(x)$  satisfies $H'>0$ on $(0,\infty)$ and $H(+0)=0$. 
Thanks to \eqref{eqh2}, for the upper bound of \eqref{eqh} it suffices to show that 
\[
\frac1{2}\left[\frac{e^{x^2}}{x} + \sqrt{\left(\frac{e^{x^2}}{x}\right)^2 - \pi }  \right] \ge  \frac{e^{x^2}-1}{x},  
\qquad  \text{or equivalently} \qquad 
\sqrt{\left(\frac{e^{x^2}}{x}\right)^2 - \pi }  \ge \frac{e^{x^2} -2}{x}. 
\]
The latter is obvious if $e^{x^2}-2 <0$. In the case $e^{x^2}-2 \ge0$ the desired inequality is equivalent to 
\[
\left(\frac{e^{x^2}}{x}\right)^2 - \pi  \ge \frac{e^{2x^2} - 4 e^{x^2}+4}{x^2}, 
\qquad  \text{or equivalently} \qquad 
4 e^{x^2} - \pi x^2 -4 \ge0. 
\]
By calculus, this is the case. Thus we are done. 

The lower bound of \eqref{eqh} can also be proved by calculus. Let 
\[
J(x) := 2 h(x) - \left[\frac{e^{x^2}}{x} - \sqrt{\left(\frac{e^{x^2}}{x}\right)^2 - \pi }  \right]. 
\] 
It is elementary to see that $J(+0)=0$. It then suffices to show that $J'>0$. To begin, we compute
\[
J'(x) = 2e^{x^2} - \left[-\frac{e^{x^2}}{x^2} + 2e^{x^2} - \frac{\frac{-2e^{2x^2}}{x^3}  + \frac{4e^{2x^2}}{x}  }{2\sqrt{\left(\frac{e^{x^2}}{x}\right)^2 - \pi }}  \right] 
= \frac{e^{2x^2}}{x^3}\left[x e^{-x^2}  + \frac{-1 + 2x^2 }{\sqrt{\left(\frac{e^{x^2}}{x}\right)^2 - \pi }} \right].
\]
Hence $J'$ is obviously positive on $[1/\sqrt{2}, \infty)$. Suppose that $J'(x)=0$ held for some $0< x < 1/\sqrt{2}$. This would imply 
\[
x^2 e^{-2x^2}\left[ \left(\frac{e^{x^2}}{x}\right)^2 - \pi \right]  = (2x^2-1)^2, 
\qquad  \text{or equivalently} \qquad 
\pi e^{-2x^2} +4x^2-4=0. 
\]
However, by some elementary calculus the left hand side of the last equation must be negative, a contradiction. We are done. 
\end{proof}

Next, we find a failure of Boolean selfdecomposablity for normal distributions.

\begin{thm}
There exists $M>0$ such that $|m/\sqrt{v}|> M$ implies that $N(m,v)\notin L(\uplus)$.
\end{thm}
\begin{proof}
It suffices to find $M>0$ and show that $N(m,1)\notin L(\uplus)$ for $|m|>M$ thanks to the fact that dilation preserves $L(\uplus)$. Recall that 
\[
\ell_{\ND}(x) = |x|k_{\ND}(x)=\frac{c}{f(x)},
\]
where $c=\pi^{-1}\sqrt{2/\pi}$ and $f$ is defined in \eqref{eq:f}. 
Because $\ND$ does not have a Boolean Gaussian component from Proposition \ref{prop:k} and the function $\ell_{\ND}$ is not constant on $(0,\infty)$ and is symmetric on $\R$, the desired conclusion follows by Theorem \ref{thm:shift}. 
\end{proof}

\begin{rem}[Estimate of $M$] According to Theorem \ref{thm:shift}, $N(m,1)$ is not in $L(\uplus)$ if $|m| > M_0$, where 
\begin{equation} \label{eq:inf}
M_0 = \inf\left\{\frac{b\ell_{N(0,1)}(a) - a \ell_{N(0,1)}(b)}{\ell_{N(0,1)}(b) -\ell_{N(0,1)}(a)}: -\infty < a < b <\infty,\ \ell_{N(0,1)}(a) < \ell_{N(0,1)}(b) \right\}.
\end{equation}
With some inspection helped by the identity
\[
\frac{b\ell_{N(0,1)}(a) - a \ell_{N(0,1)}(b)}{\ell_{N(0,1)}(b) -\ell_{N(0,1)}(a)} = \frac{(b - a)  \ell_{N(0,1)}(a) }{\ell_{N(0,1)}(b) -\ell_{N(0,1)}(a)}- a  = \frac{(b - a) \ell_{N(0,1)}(b) }{\ell_{N(0,1)}(b) -\ell_{N(0,1)}(a)} - b,  
\]
the infimum \eqref{eq:inf} turns out to be achieved in the limit $|b-a| \to 0$ and only when $a, b$ are both negative, so that we obtain 
\[
M_0  = \inf\left\{p(a) : -\infty < a <0 \right\}, \qquad  p(a):=\frac{\ell_{N(0,1)}(a)}{\ell_{N(0,1)}'(a)}-a.  
\]
In order to find the infimum of the function $p$, observe that 
\[
p'(a) =  \frac{-\ell_{N(0,1)}(a)\ell_{N(0,1)}''(a)}{\ell_{N(0,1)}'(a)^2}. 
\]
According to simulations, it is likely that $\ell_{N(0,1)}''$ in $(-\infty,0)$ has a unique zero (denoted $a_0$ below) and hence  $p$ takes its minimum at $a_0$; see Figure \ref{dia0}. The values $a_0$ and $M_0=p(a_0)$ are approximately $-2.03$ and $3.09$, respectively. Moreover, simulations also suggest that $M_0$ is the precise threshold, i.e.\ $N(m,1)$ is Boolean selfdecomposable if and only if $|m|\le M_0$; see Figures \ref{dia1}, \ref{dia2}.  Those simulations are performed on Mathematica Version 12.1.1, Wolfram Research, Inc., Champaign, IL. 
\end{rem}

\begin{figure}[htpb]
\begin{center}
\begin{minipage}{0.3\hsize}
\begin{center}
\includegraphics[width=40mm,clip]{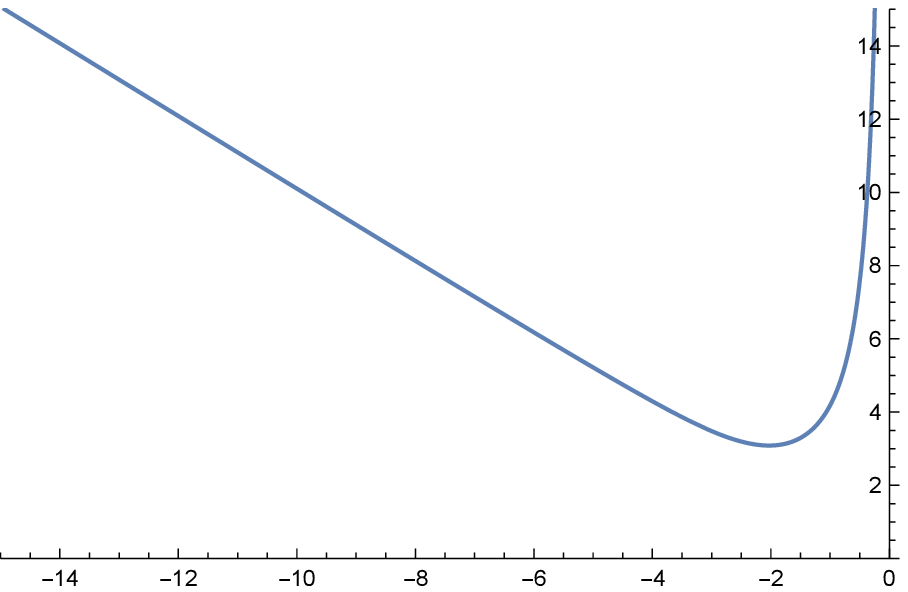}
\caption{the function $p$}\label{dia0}
\end{center}
  \end{minipage}
\begin{minipage}{0.3\hsize}
\begin{center}
\includegraphics[width=40mm,clip]{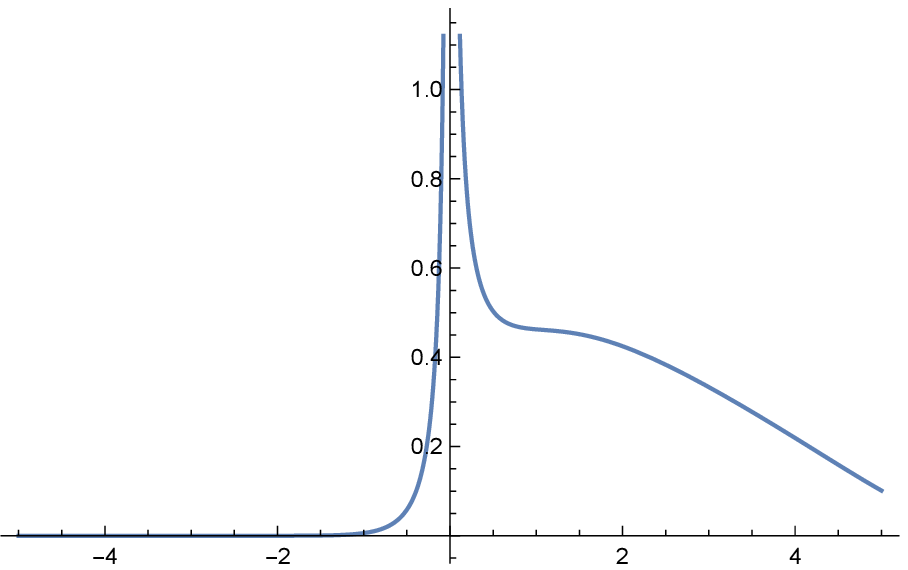}
\caption{the function $k_{N(3.05,1)}$}\label{dia1}
\end{center}
  \end{minipage}
\begin{minipage}{0.3\hsize}
\begin{center}
\includegraphics[width=40mm,clip]{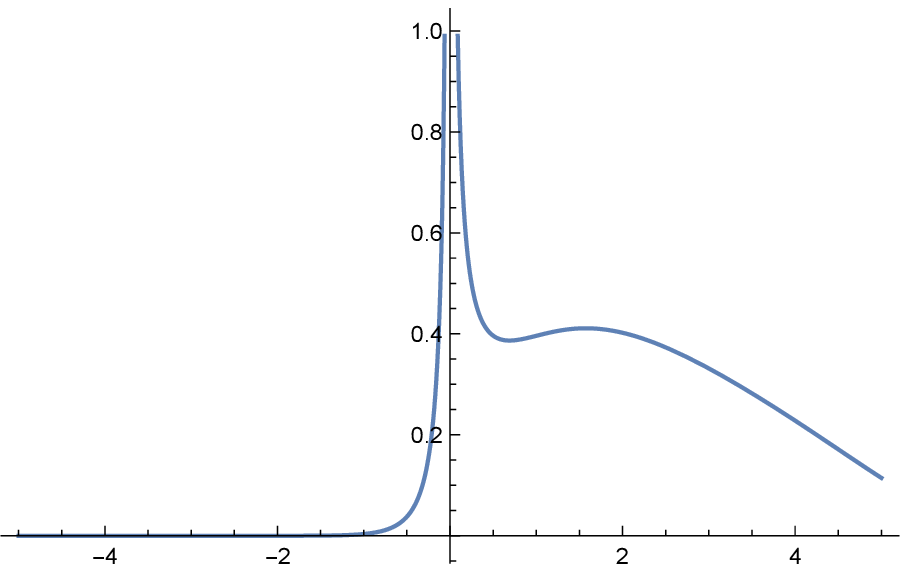}
\caption{the function $k_{N(3.2,1)}$}\label{dia2}
\end{center}
\end{minipage}
\end{center}
\end{figure}

 \section*{Acknowledgements} 
 
This research is supported by JSPS Open Partnership Joint Research Projects Grant Number JPJSBP120209921. 
Moreover, T.H.\ is supported by JSPS Grant-in-Aid for Young Scientists 19K14546; K.N.\ is supported by JSPS Grant-in-Aid for Young Scientists 21K13807; Y.U.\ is supported by JSPS Grant-in-Aid for Scientific Research (B) 19H01791 and JSPS Grant-in-Aid for Young Scientists 22K13925.


\end{document}